\documentclass[a4paper, 11pt]{amsart} 
\usepackage[
  margin=1.9cm,
  includefoot,
  footskip=30pt,
]{geometry}
\usepackage{amsmath,amssymb,amscd}
\usepackage{amsfonts}
\usepackage[english]{babel}
\usepackage[leqno]{amsmath}
\usepackage{amssymb,amsthm}
\usepackage{mathrsfs} 
\usepackage{amscd}
\usepackage{enumerate}
\usepackage{epsfig}
\usepackage{relsize}
\usepackage{layout}
\usepackage[usenames,dvipsnames]{xcolor}
\usepackage[backref=page]{hyperref}
\usepackage{tikz}
\usepackage{tikz-cd}
\hypersetup{
 colorlinks,
 citecolor=Green,
 linkcolor=Red,
 urlcolor=Blue}
\usepackage[matrix,arrow,tips,curve]{xy}
\input{xy}
\xyoption{all}
\definecolor{darkgreen}{rgb}{0.0, 0.7, 0.0}
\definecolor{purple}{rgb}{0.5, 0.0, 0.5}
\definecolor{red}{rgb}{0.8, 0.2, 0.0}

\newtheorem{thm}{Theorem}[section]
\newtheorem{bthm}{Theorem}
\newtheorem{bcor}{Corollary}
\newtheorem{lemma}[thm]{Lemma}

\newtheorem{prop}[thm]{Proposition}

\numberwithin{equation}{section}
\theoremstyle{definition}
\newtheorem{defi}[thm]{Definition}

\theoremstyle{remark}
\newtheorem{remark}[thm]{Remark}

\newcommand{\Z}{\mathbb{Z}}

\DeclareMathOperator{\Supp}{Supp}

\def \P{\mathbb{P}}

\def\I{{\mathcal J}}
\def \L{\mathcal L}
\def \E{\mathcal E}
\def \G{\mathcal G}

\def\O{\mathcal O}

\def\M0{\mathcal M^0}

\DeclareMathOperator{\Ker}{{Ker}}

\begin{document}

\title[On the classification of polarized surfaces with weakly Ulrich or aCM tangent bundles]{On the classification of polarized surfaces with weakly Ulrich or initialized aCM tangent bundles}

\author[D. Bhattacharya, O. Genc, A.F. Lopez]{Debojyoti Bhattacharya, Ozhan Genc and Angelo Felice Lopez*}

\address{\hskip -.43cm Debojyoti Bhattacharya, Chennai Mathematical Institute, H1 SIPCOT IT Park, Siruseri, Kelambakkam 603103, India. email: {\tt debojyoti7054@gmail.com; debojyoti@cmi.ac.in}}

\address{\hskip -.43cm \"Ozhan Gen\c{c}, Department of Mathematics, TED University, 06420 Ankara, Turkey.\hfill \null  \hskip 3.1cm email {\tt ozhangenc@gmail.com}}

\address{\hskip -.43cm Angelo Felice Lopez, Dipartimento di Matematica e Fisica, Universit\`a di Roma
Tre, Largo San Leonardo Murialdo 1, 00146, Roma, Italy. e-mail {\tt angelo.lopez@uniroma3.it}}


\thanks{*Partially supported by the GNSAGA group of INdAM}


\begin{abstract} 
Counterexamples to the Eisenbud-Schreyer conjecture on the existence of Ulrich bundles have appeared recently (see Anghel \cite{a1} and the third author \cite{l}). As {\it weakly Ulrich} bundles are a useful generalization of Ulrich ones, in this paper, we classify smooth surfaces $S \subseteq \P^N$ with  weakly Ulrich tangent bundle. They are Del-Pezzo surfaces of degree at least 5 and we classify the possible embeddings. An interesting feature is the appearance of two infinite series of polarizations, which is quite different from the Ulrich case \cite{bmpt, lr}. We also achieve the classification of surfaces with initialized aCM tangent bundles and, as a consequence, of surfaces with almost Ulrich tangent bundles.
\end{abstract}

\maketitle


\section{Introduction}

The study of the properties of the tangent bundle of a smooth $n$-dimensional variety $X \subseteq \P^N$ is both a classical and modern subject. Important results have been obtained in the study of positivity properties of $T_X$, such as Mori's solution \cite{mo} of the Hartshorne's conjecture: if $T_X$ is ample, then $X \cong \P^n$. Several subsequent papers have dealt with $T_X$ nef or pseudoeffective (see for example \cite{cp, dps, him} and refefences therein).

A similar study has been carried on cohomological properties of the tangent bundle. Perhaps among these properties, an interesting one emerged recently, is whether $T_X$ is aCM, or even Ulrich. The notions of Ulrich and weakly Ulrich bundle (see Definition \ref{def}) were introduced by Eisenbud and Schreyer in \cite{es} (for the theory of Ulrich bundles see also \cite{cmrpl} and references therein), where it was proposed the problem that every projective variety carries an Ulrich bundle. While the aCM condition for $T_X$ appears to be in general out of reach (see Remark \ref{tante}), the Ulrich case was classified in \cite{bmpt}, where they proved that $T_X$ is Ulrich if and only if $X$ is either the twisted cubic or the Veronese surface. This result was later extended in \cite{lr}, where also twists of the tangent bundle were considered. In light of recent results of Anghel \cite{a1,a2} and the third author \cite{l}, where the Eisenbud-Schreyer's conjecture was disproved, it has become more important to deal with weakly Ulrich bundles. 

Consider now varieties with weakly Ulrich tangent bundle. 

The case of curves being simple (there is only $(\P^1,\O_{\P^1}(a)), a \ge 2$), the first goal of this paper has been to classify surface pairs $(S,\O_S(1))$ with weakly Ulrich tangent bundle. This is achieved in the following result, where we denote by $S_r$ the blow up of the plane in $r$ points in general linear position, with exceptional divisiors $E_i$ and with $\widetilde L$ the pull-back of a line.

\begin{bthm} \hskip 3cm
\label{wu}

Let $S \subseteq \P^N$ be a smooth irreducible surface. Then $T_S$ is a weakly Ulrich vector bundle for $(S,\O_S(1))$ if and only if $(S,\O_S(1))$ is one of the following (up to permutations of the $E_i's$):

\begin{itemize} 
\item[(i)] $(\P^2, \O_{\P^2}(a))$ with $2 \le a \le 4$.
\item[(ii)] $(\P^1\times \P^1,\O_{\P^1\times\P^1}(a,b))$ with $(a,b) \in \{(1,c), (c,1), c \ge 1\} \cup\{(2,2),(2,3),(3,2),(3,3)\}$.
\item[(iii)] $(S_1, a \widetilde L-b_1E_1)$ with $(a;b_1) \in \{(2;1),(3;1),(4;1),(4;2)\} \cup \{(c+1;c), c \ge 2\}$.
\item[(iv)] $(S_2, a\widetilde L-b_1E_1-b_2E_2)$ with $(a; b_1, b_2) \in \{(3;1,1),(4;1,1),(4;2,1),(5;2,2)\}$.
\item[(v)] $(S_3, a\widetilde L-b_1E_1-b_2E_2-b_3E_3)$ with $(a;b_1,b_2,b_3)$ belonging to
$$\{(3;1,1,1),(4;1,1,1),(4;2,1,1),(5;2,2,1),(5;2,2,2)\}.$$
\item[(vi)] $(S_4, a\widetilde L-b_1E_1-b_2E_2-b_3E_3-b_4E_4)$ with $(a;b_1,b_2,b_3,b_4)$ belonging to
$$\{(3;1,1,1,1),(4;1,1,1,1),(4;2,1,1,1),(5;2,2,1,1),(5;2,2,2,1),(5;2,2,2,2),(6;2,2,2,2),(6;3,2,2,2)\}.$$
\end{itemize} 
\end{bthm}

We note that there are many more cases than the Ulrich one (that occurs only in case (i) with $a=2$) and even two infinite series of examples.

In the aCM case, despite the difficulties remarked above, we managed also to classify surface pairs with initialized aCM tangent bundle, as follows.


\begin{bthm} \hskip 3cm
\label{acm}

Let $S \subseteq \P^N$ be a smooth irreducible surface. If $(S,\O_S(1))$ is one of the following:
\begin{itemize} 
\item[(i)] $(\P^2, \O_{\P^2}(a)), a \ge 2, a \ne 3$.
\item[(ii)] $(\P^1\times \P^1,\O_{\P^1\times\P^1}(a,b))$ with $a \ge 3, b \ge 3$.
\item[(iii)] $(S_r, a\widetilde L-\sum\limits_{i=1}^r b_iE_i), 1 \le r \le 4$ with $b_i \ge 2$ and either $a \ge b_1+ 3$ when $r=1$ or $a \ge b_i+b_j+2$ when $2 \le r \le 4$, where $1 \le i \ne j \le r$
\end{itemize} 
then $T_S$ is an aCM vector bundle and it is initialized except in case (iii) for $r=4$.

Vice versa, if $T_S$ is an initialized aCM vector bundle, then $(S,\O_S(1))$ is as in (i)-(iii) with $1 \le r \le 3$.
\end{bthm}
Note that the conditions (ii)-(iii) are equivalent to $S$ being a Del Pezzo surface of degree at most $5$ with $K_S+H$ is ample (for $H \in |\O_S(1)|$).

Finally, with the two results above, one easily gets a classification for $T_S$ almost Ulrich (that is weakly Ulrich and aCM).

\begin{bcor} \hskip 3cm
\label{au}

Let $S \subseteq \P^N$ be a smooth irreducible surface. Then $T_S$ is an almost Ulrich vector bundle for $(S,\O_S(1))$ if and only if $(S,\O_S(1))$ is one of the following:
\begin{itemize} 
\item[(i)] $(\P^2, \O_{\P^2}(2))$.
\item[(ii)] $(\P^2, \O_{\P^2}(4))$.
\item[(iii)] $(\P^1\times \P^1,\O_{\P^1\times\P^1}(3,3))$.
\item[(iv)] $(S_4, 6\widetilde L-2E_1-2E_2-2E_3-2E_4)$.
\end{itemize} 
\end{bcor}
Of the above cases, $T_S$ is actually Ulrich only in case (i). Also, in case (iv), $H=-2K_S$ and $(S,\O_S(1))$ is the only pair such that  $T_S(1)$ is Ulrich by \cite[Thm.~3]{lr}.

\section{Notation, conventions and basic definitions}

Throughout the paper we work over the complex numbers. Moreover we henceforth establish the following

\subsection{Notation} 
\label{nota}

\hskip 3cm

\begin{itemize} 
\item $X \subseteq \P^N$ is a smooth irreducible variety of dimension $n \ge 1$.
\item $H$ is a very ample divisor in $|\O_X(1)|$.
\item For any sheaf $\G$ on $X$ and any $\ell \in \Z$, we set $\G(\ell)=\G(\ell H)$.  
\end{itemize} 

\subsection{Bott vanishing}

 Following \cite{to}, we give the following

\begin{defi}
A smooth variety $X$ satisfies {\it Bott vanishing} if
$$H^p (\Omega^q_X \otimes \L)=0$$
for all ample line bundles $\L$, all $q \ge 0$ and all $p > 0$.
\end{defi}

Examples of varieties satisfying Bott vanishing are (see \cite{to}) $\P^n$, smooth projective toric varieties, abelian varieties and some K3 surfaces. 

\begin{remark}
\label{bott}
All Del Pezzo surfaces of degree at least $5$ satisfy Bott vanishing. This follows by \cite[Thm.~2.1]{to} for $s=5$ and by the fact that $S$ is a toric surface for $6 \le s \le 9$ (see for example \cite[Thm.~2.4(i)]{mu}). 
\end{remark}

\subsection{Ulrich, weakly Ulrich and almost Ulrich bundles}

In literature there are two definitions of weakly Ulrich bundle. One in \cite[\S 2]{es} (see also \cite[\S 2]{j}, \cite[\S 6]{acg}) and one in \cite[Def.~3.2.11]{cmrpl}, the difference being that the second one assumes initializedness of the bundle. For our purposes of studying weakly Ulrich tangent bundles, it is better the more general definition, as follows.

\begin{defi}
\label{def}
Let $X \subseteq \P^N$ be a smooth irreducible variety of dimension $n \ge 1$ and let $H \in |\O_X(1)|$. Let $\E$ be a vector bundle on $X$. We say that $\E$ is an {\it Ulrich vector bundle} for $(X,H)$ if $H^i(\E(-p))=0, i \ge 0, 1 \le p \le n$. We say that $\E$ is a {\it weakly Ulrich vector bundle} for $(X,H)$ if the following vanishings hold:
\begin{itemize} 
\item[(i)] $H^i(\E(-t))=0$ for $1 \le i \le n$ and $t \le i-1$.
\item[(ii)] $H^i(\E(-t))=0$ for $0 \le i \le n-1$ and $t \ge i+2$.
\end{itemize} 
We say that $\E$ is an {\it almost Ulrich vector bundle} for $(X,H)$ if $\E$ is weakly Ulrich and aCM for $(X,H)$ (that is either $n=1$ or $n \ge 2 $ and $H^i(\E(j))=0$ for $1 \le i \le n-1$ and all $j \in \Z$).
\end{defi}

On a surface, for a weakly Ulrich bundle, we have the following 
 
\begin{lemma} 
\label{c1}
Let $S \subseteq \P^N$ be a smooth irreducible surface and let $\E$ be a rank $r$ weakly Ulrich bundle on $S$. Then
$$c_1(\E)H-\frac{r}{2}(K_S+H)H=h^0(\E)-h^0(\E(-1))+h^1(\E(-1)).$$
In particular, $c_1(\E)H \ge \frac{r}{2}(K_S+H)H$.
\end{lemma}
\begin{proof}
By definition of weakly Ulrich we have that $\chi(\E)=h^0(\E)$ and $\chi(\E(-1))=h^0(\E(-1))-h^1(\E(-1))$. Then, Riemann-Roch gives
$$h^0(\E)-h^0(\E(-1))+h^1(\E(-1))=\chi(\E)-\chi(\E(-1))=c_1(\E)H-\frac{r}{2}(K_S+H)H.$$
Since $h^0(\E) \ge h^0(\E(-1))$, we also get that $c_1(\E)H \ge \frac{r}{2}(K_S+H)H$.
\end{proof}

\subsection{Del Pezzo surfaces}

We establish the notation for some Del Pezzo surfaces.

Consider, for $1 \le r \le 4$, the Del Pezzo surface $S_r$ of degree $s=9-r$ obtained by the blow-up 
$$\varepsilon : S_r \to \P^2$$ 
of the points $P_i \in \P^2$ in general linear position, with exceptional divisors $E_i, 1 \le i \le r$. Let $L$ be a line in $\P^2$ and let $\widetilde L = \varepsilon^*L$. If $H$ is a very ample divisor on $S$, we can write
\begin{equation}
\label{h}
H \sim a\widetilde L-\sum\limits_{i=1}^r b_iE_i.
\end{equation}
We will assume, without loss of generality, that 
\begin{equation}
\label{bi}
b_1 \ge \ldots \ge b_r.
\end{equation}
Recall that we have (see for example \cite[Cor.~4.6]{dr}),
\begin{equation}
\label{ma}
b_r \ge 1 \ \hbox{and} \ a \ge b_1+b_2+1 \ \hbox{when} \ r \in \{2, 3, 4\}, a \ge b_1+1 \ \hbox{when} \ r = 1. 
\end{equation} 

\section{Auxiliary results on Del Pezzo surfaces}

We collect here two results on Del Pezzo surfaces that will be used in the sequel.

The first one is on the plane.

\begin{lemma}
\label{duep}
Let $P_1,\ldots, P_4 \in \P^2$ be points in general linear position and let $W$ be the reduced scheme $\{P_1,P_2\}$, $W'$ be the reduced scheme $\{P_1,\ldots, P_4\}$.
Then 
$$H^0(\I_{W/\P^2} \otimes \Omega_{\P^2}(2))=H^0(\I_{W'/\P^2} \otimes \Omega_{\P^2}(3))=0.$$
\end{lemma}
\begin{proof}
Let $Z \in \{W,W'\}$. For $\ell \in \Z$, considering the twisted Euler sequence and its restriction to $Z$, we get the commutative diagram
$$\xymatrix{& 0 \ar[d] & 0 \ar[d] & 0 \ar[d] & \\ 0 \ar[r] & \I_{Z/\P^2} \otimes \Omega_{\P^2}(\ell) \ar[d] \ar[r] & H^0(\O_{\P^2}(1)) \otimes \I_{Z/\P^2}(\ell-1) \ar[d] \ar[r] & \I_{Z/\P^2}(\ell)  \ar[r] \ar[d] & 0 & \\ 0 \ar[r] & \Omega_{\P^2}(\ell) \ar[r] \ar[d] & H^0(\O_{\P^2}(1)) \otimes \O_{\P^2}(\ell-1) \ar[r] \ar[d] & \O_{\P^2}(\ell) \ar[r] \ar[d] & 0 &  \\ 0 \ar[r] & \Omega_{\P^2}(\ell)_{|Z} \ar[r] \ar[d] & H^0(\O_{\P^2}(1)) \otimes \O_{\P^2}(\ell-1)_{|Z}  \ar[r] \ar[d] & \O_{\P^2}(\ell)_{|Z} \ar[r] \ar[d] & 0 \\ & 0 & 0 & 0. & }$$
showing that
\begin{equation}
\label{mu}
H^0(\I_{Z/\P^2} \otimes \Omega_{\P^2}(\ell))= \Ker \{\mu_{\ell}: H^0(\O_{\P^2}(1)) \otimes H^0(\I_{Z/\P^2}(\ell-1)) \to H^0(\I_{Z/\P^2}(\ell))\}.
\end{equation}
The diagram also gives that $\mu_{\ell}$ is the map induced by the multiplication map 
$$H^0(\O_{\P^2}(1)) \otimes H^0(\O_{\P^2}(\ell-1)) \to H^0(\O_{\P^2}(\ell)).$$ 
In all cases for $Z$, up to a projectivity, we can assume that 
$$P_1=[1:0:0], P_2=[0:1:0], P_3=[0:0:1], P_4=[1:1:1] \in \P^2.$$ 
When $Z=W, \ell=2$ (respectively $Z=W', \ell=3$) we show that $\mu_{\ell}$ is injective, hence the required vanishing follows by \eqref{mu}. Now, for $Z=W$ and $\ell=2$, we have that $H^0(\I_{W/\P^2}(1))$ is generated by $x_2$ and $\mu_2(x_i \otimes x_2)=x_ix_2$ for $0 \le i \le 2$. Thus, $\mu_2$ is injective, proving the lemma for $W$. Similarly, for $Z=W'$ and $\ell=3$, we have that $H^0(\I_{W'/\P^2}(2))$ is generated by $x_2(x_0-x_1)$ and $x_1(x_0-x_2)$ and
$$\mu_3(x_i \otimes x_2(x_0-x_1))=x_ix_2(x_0-x_1), \mu_3(x_i \otimes x_1(x_0-x_2))=x_ix_1(x_0-x_2)$$ for $0 \le i \le 2$. Since the polynomials $\{x_ix_2(x_0-x_1), x_ix_1(x_0-x_2), 0 \le i \le 2\}$ are linearly independent, $\mu_3$ is also injective and the lemma is proved also for $W'$.
\end{proof}
The second result is on the blow-ups of the plane.

\begin{lemma}
\label{dp1}
Let $S=S_r, 1 \le r \le 4$ and let $M = \alpha \widetilde L-\sum\limits_{i=1}^r \beta_iE_i$. If $\alpha \ge -1$, then $H^2((\varepsilon^*\Omega_{\P^2})(M))=0$. 
\end{lemma}
\begin{proof}
Assume, without loss of generality, that $\beta_1 \ge \ldots \ge \beta_r$. For some $q \in \{0, \ldots, r\}$ we have $\beta_q \ge 0$, $\beta_{q+1} < 0$. Hence, we can write $M = M' + D$, where $M' =\alpha \widetilde L-\sum\limits_{i=1}^q \beta_iE_i$ and $D=\sum\limits_{i=q+1}^r (-\beta_i)E_i$. Then the exact sequence
$$0 \to (\varepsilon^*\Omega_{\P^2})(M') \to (\varepsilon^*\Omega_{\P^2})(M) \to (\varepsilon^*\Omega_{\P^2})(M)_{|D} \to 0$$
shows that we can assume that $\beta_r \ge 0$. Now let $E=\sum\limits_{i=1}^r \beta_iE_i$. Then we have an exact sequence
$$0 \to \O_S(-E) \to \O_S \to \O_E \to 0$$
whose direct image under $\varepsilon$ gives the exact sequence
$$0 \to \varepsilon_*\O_S(-E) \to \O_{\P^2} \to \O_W \to R^1 \varepsilon_*\O_S(-E) \to 0$$
where $W \subset \P^2$ is a $0$-dimensional scheme of length $\beta_i$ at $P_i, 1 \le i \le r$. In particular, we deduce that $\varepsilon_*\O_S(-E) \cong \I_{W/\P^2}$ and $R^j \varepsilon_*\O_S(-E)=0$ for $j > 0$. Therefore, Leray's formula gives that
\begin{equation}
\label{van}
H^2((\varepsilon^*\Omega_{\P^2})(M)) \cong H^2(\I_{W/\P^2} \otimes \Omega_{\P^2}(\alpha)).
\end{equation} 
Now, in the exact sequence
$$0 \to \I_{W/\P^2} \otimes \Omega_{\P^2}(\alpha) \to \Omega_{\P^2}(\alpha) \to \Omega_{\P^2}(\alpha)_{|W} \to 0$$
we have that $H^1(\Omega_{\P^2}(\alpha)_{|W})=H^2(\Omega_{\P^2}(\alpha))=0$ by Bott's formula (see for example \cite[page 8]{oss}), hence also $H^2(\I_{W/\P^2} \otimes \Omega_{\P^2}(\alpha))=0$ and we conclude using \eqref{van}.
\end{proof}

\section{$T_S$ weakly Ulrich}

In this section we study surfaces with weakly Ulrich tangent bundle. 

We show that $S$ must be a Del Pezzo surface and the condition actually amounts to a single vanishing.

\begin{prop}
\label{wU}
Let $S \subseteq \P^N$ be a smooth projective surface. Then $T_S$ is a weakly Ulrich vector bundle if and only if $S$ is a Del Pezzo surface of degree $s \ge 5$ satisfying $H^2(T_S(-1))=0$ and $(S,\O_S(1)) \not\cong (\P^2,\O_{\P^2}(1))$.
\end{prop}
\begin{proof}
If $T_S$ is weakly Ulrich, then $H^2(T_S(-1))=0$ by definition. Also, $(S,\O_S(1)) \not\cong (\P^2,\O_{\P^2}(1))$ because $h^1(T_{\P^2}(-3))=h^{1,1}(\P^2)=1$, contradicting the weakly Ulrich condition. If $H \in |\O_X(1)|$, we get from Lemma \ref{c1} that $2K_SH+H^2 \le 0$, hence $K_SH<0$. By definition of weakly Ulrich, we also have that $H^1(T_S)=0$ and it follows from \cite[Thm.~1.3]{bc} that $S$ is a Del Pezzo surface of degree $s \ge 5$. 

Vice versa, assume that $S \subseteq \P^N$ is a Del Pezzo surface of degree $s \ge 5$ with $H^2(T_S(-1))=0$ and $(S,\O_S(1)) \not\cong (\P^2,\O_{\P^2}(1))$. Let $C \in |H|$ be a smooth irreducible curve. From the exact sequence
$$0 \to T_S(t) \to T_S(t+1) \to T_S(t+1)_{|C} \to 0$$
we get by induction that 
$$H^2(T_S(t))=0 \ \hbox{for} \ t \ge -1.$$ 
Next, observe that, as is well known, $K_S+tH$ is ample for $t \ge 3$. Also, $S$ satisfies Bott vanishing by Remark \ref{bott}. Therefore, for $t \ge 3$,
$$h^1(T_S(-t))=h^1(\Omega_S(K_S+tH))=0$$
by Bott vanishing. Since $(S,\O_S(1)) \not\cong (\P^2,\O_{\P^2}(1))$, it follows by the Mori-Sumihiro-Wahl's theorem \cite[Thm.~8]{ms}, \cite[Thm.~1]{w}, that $H^0(T_S(-1))=0$, hence also $H^0(T_S(-2))=0$. Now, $H^1(T_S)=0$ by \cite[Thm.~1.3]{bc}, and therefore $T_S(1)$ is $0$-regular with respect to $H$ by the hypothesis $H^2(T_S(-1))=0$. Hence, we get by Castelnuovo-Mumford that 
$$H^1(T_S(t))=0 \ \hbox{for} \ t \ge 0.$$ 
Therefore, $T_S$ is weakly Ulrich.
\end{proof}

\begin{remark}
\label{-h}
If $S$ is a Del Pezzo surface of degree $s \ge 5$, then $T_S$ is a weakly Ulrich vector bundle for $(S,-K_S)$. Indeed, $h^2(T_S(-1))=h^0(\Omega_S(K_S+H))=h^0(\Omega_S)=0$ and Proposition \ref{wU} applies.
\end{remark}

With Proposition \ref{wU} at hand, it is easy to classify the case of minimal Del Pezzo surfaces.

\begin{prop}
\label{min}
Let $a, b \in \Z, a \ge 1, b \ge 1$. 

\begin{itemize}
\item[(i)] If $(S,\O_S(1)) \cong (\P^2,\O_{\P^2}(a))$, then $T_S$ is weakly Ulrich if and only if $2 \le a \le 4$.
\item[(ii)] If $(S,\O_S(1)) \cong (\P^1\times \P^1,\O_{\P^1\times\P^1}(a,b))$, then $T_S$ is weakly Ulrich if and only if either $a=1$, or $b=1$, or $a, b \in \{2,3\}$.
\end{itemize} 
\end{prop}
\begin{proof}
By Proposition \ref{wU}, $T_S$ is weakly Ulrich if and only if $(S,\O_S(1)) \not\cong (\P^2,\O_{\P^2}(1))$ and $h^2(T_S(-1))=h^0(\Omega_S(K_S+H))=0$. If $(S,\O_S(1)) \cong (\P^2,\O_{\P^2}(a))$, $T_S$ is weakly Ulrich if and only if $a \ge 2$ and $h^0(\Omega_{\P^2}(a-3))=0$, that is $a \le 4$ by Bott's formula (see for example \cite[page 8]{oss}).

If $(S,\O_S(1)) \cong (\P^1\times \P^1,\O_{\P^1\times\P^1}(a,b))$, let $\pi_i : S \cong \P^1\times \P^1 \to \P^1, i \in \{1,2\}$ be the two projections. Then K\"unneth's formula gives 
$$\begin{aligned}
H^0(\Omega_S(K_S+H)) & =H^0(\Omega_{\P^1\times \P^1}(a-2,b-2))=H^0(\pi_1^*\Omega_{\P^1}(a-2,b-2) \oplus \pi_2^*\Omega_{\P^1}(a-2,b-2)) \\
& = \big(H^0(\O_{\P^1}(a-4)) \otimes H^0(\O_{\P^1}(b-2))\big) \oplus \big(H^0(\O_{\P^1}(a-2)) \otimes H^0(\O_{\P^1}(b-4))\big)=0
\end{aligned}$$
if and only if either $a=1$, or $b=1$, or $a, b \in \{2,3\}$. 
\end{proof}
Next, in order to study non-minimal Del Pezzo surfaces, we will collect some useful consequences of Proposition \ref{wU}. 

\begin{lemma}
\label{nec}
Let $S=S_r, 1 \le r \le 4$, let $H$ be as in \eqref{h}, \eqref{bi} and \eqref{ma} and let $s=9-r=K_S^2$. 

If $T_S$ is weakly Ulrich, the following hold:
\begin{itemize}
\item [(i)] $H^2+2HK_S+2s-10+h^1(T_S(-1))=0.$

\noindent In particular, $H^2+2HK_S+2s-10 \le 0$, with equality if $K_S+H$ is ample.
\item [(ii)] Let $E$ be an effective divisor such that $K_S+H-E$ is ample and set $D=H-E$. Then 
$$D^2+2DK_S+2s-10=0.$$
\item [(iii)] $H^0((a-5)\widetilde L-\sum\limits_{i=1}^r (b_i-1)E_i+2E_1)=0$.
\item [(iv)] If $r=4$, then $H^0((a-7)\widetilde L-\sum\limits_{i=1}^4 (b_i-3)E_i)=0$.
\item [(v)] $H^0(\I_{Z/\P^2} \otimes \Omega_{\P^2}(a-3))=0$, where $Z$ is the $0$-dimensional subscheme of $\P^2$ supported at the points $P_i$ with multiplicity $b_i-1, 1 \le i \le r$ (that is with ideal sheaf $\I_{Z/\P^2} = \varepsilon_*(-\sum\limits_{i=1}^r (b_i-1)E_i)$).
\end{itemize}

Vice versa, $T_S$ is weakly Ulrich if one of the following holds:
\begin{itemize}
\item [(vi)] $r=1, H^0((a-5)\widetilde L-(b_1-3)E_1)=0$ and $a \le b_1+3$, or
\item [(vii)] $H^0(\I_{Z/\P^2} \otimes \Omega_{\P^2}(a-3))=0$ and $b_1 \le 2, 1 \le i \le r$, where $Z$ is as in (v).
\end{itemize}
\end{lemma}
\begin{proof} 
To prove (i)-(v), assume that $T_S$ is weakly Ulrich. Then Serre's duality gives that 
$$h^2(T_S(-1))=h^0(\Omega_S(K_S+H))=0.$$ 
Moreover, $S$ satisfies Bott vanishing by Remark \ref{bott} and $H^0(T_S(-1))=0$ by the Mori-Sumihiro-Wahl's theorem \cite[Thm.~8]{ms}, \cite[Thm.~1]{w}.

To see (i), we have by Lemma \ref{c1} that
\begin{equation}
\label{rr}
h^0(T_S)+h^1(T_S(-1))+2HK_S+H^2=0.
\end{equation}
Now, $H^i(T_S)=0$ for $i \ge 1$ because $T_S$ is weakly Ulrich and $h^0(T_S)=\chi(T_S)=2s-10$ by Riemann-Roch. Replacing into \eqref{rr} gives the equality and the inequality in (i). If $K_S+H$ is ample, then $h^1(T_S(-1))=h^1(\Omega_S(K_S+H))=0$ by Bott vanishing and \eqref{rr} gives that $H^2+2HK_S+2s-10=0$. This proves (i). Next, consider (ii). Since $K_S+D$ is ample, we get by Bott vanishing that $h^i(T_S(-D))=h^{2-i}(\Omega_S(K_S+D))=0$ for $i \le 1$. On the other hand, 
$$h^2(T_S(-D))=h^0(\Omega_S(K_S+D))=h^0(\Omega_S(K_S+H-E)) \le h^0(\Omega_S(K_S+H))=0.$$
Hence $\chi(T_S(-D))=0$ and Riemann-Roch gives (ii). As for (iii) and (iv), consider the morphism $f: S \to \P^1$ given by the base-point-free pencil $|\L|$, with $\L=\widetilde L-E_1$ in case (iii) and $\L=2\widetilde L-E_1-E_2-E_3-E_4$ in case (iv). We obtain an exact sequence
\begin{equation}
\label{rel}
0 \to T_{S/\P^1} \to T_S  \to f^*T_{\P^1} \to \G \to 0
\end{equation}
where $\G$ is a sheaf on $S$ with $\dim \Supp(\G) \le 1$ by generic smoothness and $\G=0$ when $r=1$ because $f$ is smooth. Observing that $f^*T_{\P^1} \cong \O_S(2\L)$ and twisting \eqref{rel} by $\O_S(-H)$, we get the exact sequence
\begin{equation}
\label{rel2}
0 \to T_{S/\P^1}(-H) \to T_S(-1) \to \O_S(2\L-H) \to \G' \to 0
\end{equation}
where $\G'$ is a sheaf on $S$ with $\dim \Supp(\G') \le 1$ and $\G'=0$ when $r=1$. Since $H^2(T_S(-1))=H^2(\G')=0$, we deduce from \eqref{rel2} that $H^2(\O_S(2\L-H))=0$, which, replacing the two expressions of $\L$, is (iii) and (iv) by Serre's duality. To see (v), recall that 
$$\Omega_{S/\P^2} \cong \bigoplus\limits_{i=1}^r \Omega_{E_i}.$$
Then we have an exact sequence
\begin{equation}
\label{omega}
0 \to (\varepsilon^*\Omega_{\P^2})(K_S+H) \to \Omega_S(K_S+H) \to \bigoplus\limits_{i=1}^r \O_{\P^1}(b_i-3) \to 0.
\end{equation}
Since $H^0(\Omega_S(K_S+H))=0$, we deduce from \eqref{omega} that
$$0=H^0((\varepsilon^*\Omega_{\P^2})(K_S+H))=H^0(\varepsilon^*(\Omega_{\P^2}(a-3))(-\sum\limits_{i=1}^r (b_i-1)E_i))=H^0(\I_{Z/\P^2} \otimes \Omega_{\P^2}(a-3))$$
showing (v). Next, to see (vi), observe that $T_{S/\P^1} \cong \O_S(\widetilde L+E_1)$ by \eqref{rel}. Hence 
$$h^2(\O_S(\widetilde L+E_1-H)=h^0(\O_S((a-4)\widetilde L-b_1E_1))=0$$ 
because $\big((a-4)\widetilde L-b_1E_1)\big)(\widetilde L-E_1)=a-4-b_1<0$ and $\widetilde L-E_1$ is nef. On the other hand, the vanishing in (vi) is equivalent, by Serre's duality, to $H^2(\O_S(2\widetilde L-2E_1-H))=0$. Therefore, \eqref{rel2} gives that $H^2(T_S(-1))=0$ and Proposition \ref{wU} implies that $T_S$ is weakly Ulrich. Similarly, (vii) follows from \eqref{omega}, Proposition \ref{wU} and Serre's duality, just by observing that the vanishing in (vii) is equivalent to $H^0(\varepsilon^*\Omega_{\P^2}(K_S+H))=0$ and the condition $b_1 \le 2$ implies that $H^0(\bigoplus\limits_{i=1}^r \O_{\P^1}(b_i-3) )=0$.
\end{proof}

\section{Proof of Theorem \ref{wu}}

In this section we will prove Theorem \ref{wu}.

\renewcommand{\proofname}{Proof of Theorem \ref{wu}}
\begin{proof}
By Proposition \ref{wU}, we need to consider for $S$ the cases $\P^2, \P^1 \times \P^1$ or $S_r, 1 \le r \le 4$.

The first two cases are done in Proposition \ref{min}. We therefore study the remaining cases
$$(S,H)=(S_r,a\widetilde L-\sum\limits_{i=1}^r b_iE_i)$$
under the conditions \eqref{bi} and \eqref{ma}.

\underline{\bf Case 1: $r=1$}.

By \eqref{ma}, we have
$$b_1 \ge 1, a \ge b_1+1.$$
Assume that $T_S$ is weakly Ulrich. 

By Lemma \ref{nec}(i) we find that
$$a^2-b_1^2-6a+2b_1+6 \le 0$$
But the left-hand side is positive if $b_1 \ge 2$ and $a \ge b_1+3$. Thus, we are left with the following possibilities:
\begin{itemize}
\item [(1)] $b_1=1, a \ge 2$.
\item [(2)] $b_1 \ge 2, b_1+1 \le a \le b_1+2$. 
\end{itemize}
First, consider case (2). If $a = b_1+2$, then Lemma \ref{nec}(iii) implies that $H^0((b_1-3)(\widetilde L-E_1))=0$, hence $b_1=2$, giving rise to the case $(a;b_1)=(4;2)$. If $a = b_1+1$, we are in the case $(a;b_1)=(c+1;c), c \ge 2$.

In case (1), we need to show that $a \le 4$. But if $a \ge 5$, then $K_S+H-E_1=(a-3)\widetilde L-E_1$ is ample by \cite[Cor.~4.6]{dr}. Let $D=H-E_1=a\widetilde L-2E_1$. Then applying Lemma \ref{nec}(ii) to $D$, we get $0=a^2-6a+6 \ge 1$, a contradiction and we are done in this case.

This proves that $(a;b_1) \in \{(2;1),(3;1),(4;1),(4;2)\} \cup \{(c+1;c), c \ge 2\}$.

Conversely, suppose that $(a;b_1) \in \{(2;1),(3;1),(4;1),(4;2)\} \cup \{(c+1;c), c \ge 2\}$.

Then $a \le b_1+3$ and we claim that $H^0((a-5)\widetilde L-(b_1-3)E_1)=0$: the latter is clear if $a \le 4$. If $a \ge 5$, then $(a;b_1)=(c+1;c), c \ge 4$ and required the vanishing becomes $H^0((c-4)\widetilde L-(c-3)E_1)=0$. The latter clearly holds is since $((c-4)\widetilde L-(c-3)E_1)(\widetilde L-E_1)=-1$ and $\widetilde L-E_1$ is nef. Therefore, $T_S$ is weakly Ulrich by Lemma \ref{nec}(vi).

This concludes the case $r=1$.

\underline{\bf Case 2: $r=2$}.

By \eqref{bi} and \eqref{ma}, we have
$$b_1  \ge b_2 \ge 1, a \ge b_1+ b_2+1.$$
Assume that $T_S$ is weakly Ulrich. 

By Lemma \ref{nec}(i) we find that
\begin{equation}
\label{r=2}
a^2-b_1^2-b_2^2-6a+2b_1+2b_2+4 \le 0.
\end{equation}
But the left-hand side in \eqref{r=2} is positive if $b_1b_2 \ge 3, a \ge b_1+b_2+2$. Also note that, if $a=b_1+b_2+1$, then \eqref{r=2} becomes
$$(b_1-1)(b_2-1)\le 1.$$
Thus, we are left with the following possibilities: 
\begin{itemize}
\item [(3)] $(a;b_1,b_2)=(5;2,2)$.
\item [(4)] $b_2=1,b_1 \ge 3, a=b_1+2$. 
\item [(5)] $b_2=1,1 \le b_1 \le 2, a \ge b_1+2$. 
\end{itemize}
Now we apply Lemma \ref{nec}(iii). In case (4) we get that $H^0((b_1-3)(\widetilde L-E_1))=0$, a contradiction. In case (5), if $b_1=1$, we have that $H^0((a-5)\widetilde L+2E_1))=0$, hence $3 \le a \le 4$, while if $b_1=2$, we find that $H^0((a-5)\widetilde L+E_1))=0$, hence $a=4$. 

Therefore, $(a;b_1,b_2) \in \{(3;1,1),(4;1,1),(4;2,1),(5;2,2)\}$.

Conversely, assume that $(a;b_1,b_2) \in \{(3;1,1),(4;1,1),(4;2,1),(5;2,2)\}$.

In the case $(3;1,1)$ we have that $H=-K_S$, hence $T_S$ is weakly Ulrich by Remark \ref{-h}.

In the other cases, we have $b_2 \le b_1 \le 2$ and we claim that  
\begin{equation}
\label{va}
H^0(\I_{Z/\P^2} \otimes \Omega_{\P^2}(a-3))=0.
\end{equation}
In the cases $(4;1,1)$ and $(4;2,1)$, the vanishing in \eqref{va} is easy:
$$H^0(\I_{Z/\P^2} \otimes \Omega_{\P^2}(1)) \subseteq H^0(\Omega_{\P^2}(1))=0.$$
In the case $(5;2,2)$, the vanishing in \eqref{va} follows from Lemma \ref{duep}.

Therefore, we have shown that \eqref{va} holds in all cases $(4;1,1),(4;2,1),(5;2,2)$ and Lemma \ref{nec}(vii) implies that $T_S$ is weakly Ulrich.

This concludes the case $r=2$.

\underline{\bf Case $3$: $r=3$.}

By \eqref{bi} and \eqref{ma}, we have
$$b_1  \ge b_2 \ge b_3 \ge 1, a \ge b_1+ b_2+1.$$
Assume that $T_S$ is weakly Ulrich. 

By Lemma \ref{nec}(i) we find that
\begin{equation}
\label{r=3}
a^2-b_1^2-b_2^2-b_3^2-6a+2b_1+2b_2+2b_3+2 \le 0.
\end{equation}
But the left-hand side in \eqref{r=3} is positive if $a \ge b_1+b_2+2$ and either $b_3 \ge 2$ or $b_3=1$ and $b_1b_2 \ge 3$. Thus, we are left with the following possibilities: 
\begin{itemize}
\item [(6)] $b_3 \ge 2, a= b_1+b_2+1$. 
\item [(7)] $b_3=1, b_1 \ge b_2 \ge 2, a= b_1+b_2+1$. 
\item [(8)] $b_3=b_2=1, b_1 \ge 3, a= b_1+b_2+1$. 
\item [(9)] $b_3=b_2=1, 1 \le b_1 \le 2, a \ge b_1+b_2+1$. 
\end{itemize}
In case (6), assume first that $b_2>b_3$. Then 
$$K_S+H-(\widetilde L-E_1-E_2)=(b_1+b_2-3)\widetilde L-(b_1-2)E_1-(b_2-2)E_2-(b_3-1)E_3$$
is ample by \cite[Cor.~4.6]{dr}. Let $D=H-\widetilde L+E_1+E_2=(b_1+b_2)\widetilde L-(b_1-1)E_1-(b_2-1)E_2-b_3E_3$. Then Lemma \ref{nec}(ii) gives $D^2+2DK_S+2=0$, that is
$$2b_1b_2-2b_1-2b_2+2b_3-b_3^2-4=0.$$
Since $b_1 \ge b_2 \ge b_3+1$, we deduce from the above that 
$$0=2b_1b_2-2b_1-2b_2+2b_3-b_3^2-4 \ge (b_3+1)^2+b_1b_2-2b_1-2b_2+2b_3-b_3^2-4=(b_1-2)(b_2-2)+4b_3-7>0$$
a contradiction. Therefore, still in case (6), we have that $b_2=b_3$. Now Lemma \ref{nec}(iii) gives
$$H^0((b_1-3)(\widetilde L-E_1)+(b_2-1)(\widetilde L-E_2-E_3))=0$$
that is possible only if $b_1=2$. This gives the case $(a;b_1,b_2,b_3)=(5;2,2,2)$ and concludes the study of case (6).

In cases (7) and (8), Lemma \ref{nec}(iii) gives
$$H^0((b_1-3)(\widetilde L-E_1-E_2)+(b_1-3)E_2+(b_2-1)(\widetilde L-E_2))=0$$
that is a contradiction in case (8), while leaves only the case $b_1=2$ in case (7), leading to $(a;b_1,b_2,b_3)=(5;2,2,1)$.

In case (9), Lemma \ref{nec}(iii) gives
\begin{equation}
\label{nuo}
H^0((a-b_1-4)\widetilde L+(b_1-1)(\widetilde L-E_1)+2E_1)=0
\end{equation}
and therefore $b_1+2 \le a \le b_1+3$. If $b_1=1$, we get cases $(a;b_1,b_2,b_3) \in \{(3;1,1,1), (4;1,1,1)\}$. If $b_1=2$, then either $(a;b_1,b_2,b_3)=(4;2,1,1)$ or $(a;b_1,b_2,b_3)=(5;2,1,1)$. But the latter is excluded by \eqref{nuo}. 

Therefore, $(a;b_1,b_2,b_3) \in \{(3;1,1,1),(4;1,1,1),(4;2,1,1),(5;2,2,1),(5;2,2,2)\}$.

Conversely, assume that $(a;b_1,b_2,b_3) \in \{(3;1,1,1),(4;1,1,1),(4;2,1,1),(5;2,2,1),(5;2,2,2)\}$.

In the case $(3;1,1,1)$ we have that $H=-K_S$, hence $T_S$ is weakly Ulrich by Remark \ref{-h}.

In the other cases, observe that $b_1 \le 2$ and we will apply Lemma \ref{nec}(vii). In the cases $(a;b_1,b_2,b_3) \in \{(4;1,1,1),(4;2,1,1)\}$ we have that $K_S+H$ is either $\widetilde L$ or $\widetilde L-E_1$, so that either $Z=\{P_1\}$ (reduced), or $Z=\emptyset$, and in both cases
$$H^0(\I_{Z/\P^2} \otimes \Omega_{\P^2}(1)) \subseteq H^0(\Omega_{\P^2}(1))=0.$$
If $(a;b_1,b_2,b_3)=(5;2,2,1)$ we have that $Z=\{P_1,P_2\}$ (reduced), hence $H^0(\I_{Z/\P^2} \otimes \Omega_{\P^2}(2))=0$ by Lemma \ref{duep}. Finally, if $(a;b_1,b_2,b_3)=(5;2,2,2)$, then $Z=\{P_1,P_2,P_3\}$ (reduced) and if $Z'=\{P_1,P_2\}$ (reduced), Lemma \ref{duep} implies that
$$H^0(\I_{Z/\P^2} \otimes \Omega_{\P^2}(2)) \subseteq H^0(\I_{Z'/\P^2} \otimes \Omega_{\P^2}(2))=0.$$
Therefore, in all cases Lemma \ref{nec}(vii) applies and $T_S$ is weakly Ulrich. 

This concludes the case $r=3$.

\underline{\bf Case $4$: $r=4$.}

By \eqref{bi} and \eqref{ma}, we have
$$b_1  \ge b_2 \ge b_3 \ge b_4 \ge 1, a \ge b_1+ b_2+1.$$
Assume that $T_S$ is weakly Ulrich. 

By Lemma \ref{nec}(i) we find that
\begin{equation}
\label{r=4}
a^2-b_1^2-b_2^2-b_3^2-b_4^2-6a+2b_1+2b_2+2b_3+2b_4 \le 0.
\end{equation}
An elementary but long calculation shows (see Lemma \ref{1}) that we are left with the following possibilities: 
\begin{itemize}
\item [(10)] $(5;1,1,1,1)$. 
\item [(11)] $(6;2,2,2,2)$.  
\item [(12)] $(b_1+3;b_1,1,1,1), 1 \le b_1 \le 3$. 
\item [(13)] $(b_1+3;b_1,2,1,1), 2 \le b_1 \le 3$. 
\item [(14)] $(8;4,3,3,3)$. 
\item [(15)] $(b_1+3;b_1,2,2,2), 3 \le b_1 \le 4$.  
\item [(16)] $(b_1+3;b_1,2,2,1), 3 \le b_1 \le 4$. 
\item [(17)] $(b_1+2;b_1,1,1,1), b_1 \ge 2$. 
\item [(18)] $(9;4,4,4,3)$. 
\item [(19)] $(7;3,3,3,2)$.  
\item [(20)] $(2b_1+1;b_1,b_1,b_1,1), 2 \le b_1 \le 3$. 
\item [(21)] $(2b_1+1;b_1,b_1,b_1,b_1), b_1 \ge 1$. 
\end{itemize}
Consider, respectively, the cases (10), (12) with $b_1 \in \{2,3\}$, (13) with $b_1=3$, (15) with $b_1=4$, (16), (17) with $b_1 \ge 3$ and (20) with $b_1=3$. Then Lemma \ref{nec}(iii) gives, respectively, that the following cohomology groups vanish: $H^0(2E_1), H^0((b_1-2)\widetilde L-(b_1-3)E_1), H^0(\widetilde L-E_2), H^0(2\widetilde L-E_1-E_2-E_3-E_4), H^0((b_1-2)\widetilde L-(b_1-3)E_1-E_2-E_3), H^0((b_1-3)(\widetilde L-E_1)), H^0(2(\widetilde L-E_2-E_3))$. Since all of these groups do not vanish, we have excluded the corresponding cases.

Consider now, respectively, the cases (14), (18), (19) and (21) with $b_1 \ge 3$. Then Lemma \ref{nec}(iv) gives, respectively, that the following cohomology groups vanish: $H^0(\widetilde L-E_1), H^0(2\widetilde L-E_1-E_2-E_3), H^0(E_4), H^0((b_1-3)(2\widetilde L-E_1-E_2-E_3-E_4))$. Since all of these groups do not vanish, we have excluded the corresponding cases.


We have therefore obtained list in (vi) of the theorem.

To conclude, consider $(a;b_1,b_2,b_3,b_4)$ as in (vi) of the theorem.

In the case $(3;1,1,1,1)$ we have that $H=-K_S$, hence $T_S$ is weakly Ulrich by Remark \ref{-h}.

Consider the remaining cases, with the exception of $(6;3,2,2,2)$. Observe that $b_1 \le 2$ and we will show that $T_S$ is weakly Ulrich by applying Lemma \ref{nec}(vii) and verifying that
\begin{equation}
\label{va2}
H^0(\I_{Z/\P^2} \otimes \Omega_{\P^2}(a-3))=0.
\end{equation}
In the cases $(4;1,1,1,1)$ and $(4;2,1,1,1)$ we have that $H^0(\I_{Z/\P^2} \otimes \Omega_{\P^2}(1)) \subseteq H^0(\Omega_{\P^2}(1))=0$, so that \eqref{va2} holds.

In the cases $(5;2,2,1,1), (5;2,2,2,1)$ and $(5;2,2,2,2)$, we have that $Z$ is, respectively, $Z_1=\{P_1,P_2\} \subset Z_2=\{P_1,P_2,P_3\} \subset Z_3=\{P_1,P_2,P_3,P_4\}$ (all reduced). Hence
$$H^0(\I_{Z_3/\P^2} \otimes \Omega_{\P^2}(2)) \subseteq H^0(\I_{Z_2/\P^2} \otimes \Omega_{\P^2}(2)) \subseteq H^0(\I_{Z_1/\P^2} \otimes \Omega_{\P^2}(2))=0$$ 
by Lemma \ref{duep} and again \eqref{va2} holds.

In the case $(6;2,2,2,2)$, we have that $Z=\{P_1,P_2,P_3,P_4\}$ (reduced). Hence
$H^0(\I_{Z/\P^2} \otimes \Omega_{\P^2}(3))=0$ by Lemma \ref{duep} and again \eqref{va2} holds.

Finally, to prove that $T_S$ is weakly Ulrich in case $(6;3,2,2,2)$, observe that $K_S+H=-K_S-E_1$ and therefore
$$h^2(T_S(-1))=h^0(\Omega_S(K_S+H))=h^0(\Omega_S(-K_S-E_1)) \le h^0(\Omega_S(-K_S))=h^0(T_S)=0$$
and therefore $T_S$ is weakly Ulrich by Proposition \ref{wU}.

This concludes the case $r=4$.

We have therefore obtained the complete list of $(r+1)$-tuples $(a;b_1,\ldots,b_r)$ in the theorem.
\end{proof}
\renewcommand{\proofname}{Proof}

\section{$T_S$ initialized aCM and proof of Theorem \ref{acm}.}

We first consider the initialized aCM condition for the tangent bundle of a surface $S$. 

\begin{prop}
\label{aU2}
Let $S \subseteq \P^N$ be a smooth projective surface. If $T_S$ is an initialized aCM vector bundle, then $S$ is a Del Pezzo surface of degree at least $6$.
\end{prop}
\begin{proof}
Since $T_S$ is aCM, we have that $H^1(T_S)=0$, hence it follows from \cite[Thm.~1.3]{bc} that $S$ is either a Del Pezzo surface of degree at least $5$ or is a minimal surface of general type. Also, $H^0(T_S)=0$ both on a Del Pezzo surface of degree $5$ and, as is well known, on a minimal surface of general type. Hence the latter two cases are excluded by the initialized condition.
\end{proof}

Now, we consider $\P^2$ and $\P^1 \times \P^1$.

\begin{prop}
\label{min2}
Let $a, b \in \Z, a \ge 1, b \ge 1$. 
\begin{itemize}
\item[(i)] If $(S,\O_S(1)) \cong (\P^2,\O_{\P^2}(a))$, then $T_S$ is aCM if and only if $a \ge 2, a \ne 3$.
\item[(ii)] If $(S,\O_S(1)) \cong (\P^1\times \P^1,\O_{\P^1\times\P^1}(a,b))$, then $T_S$ is aCM if and only if $a \ge 3, b \ge 3$. 
\end{itemize} 
Moreover, in both cases, $T_S$ is initialized.
\end{prop}
\begin{proof}
If $(S,\O_S(1)) \cong (\P^2,\O_{\P^2}(a))$, $T_S$ is not aCM if and only if there exists $j \in \Z$ such that 
$$0 \ne h^1(T_S(j))=h^1(T_{\P^2}(ja)=h^1(\Omega_{\P^2}(-3-ja))$$
which, by Bott's formula (see for example \cite[page 8]{oss}) is equivalent to
$$\exists j \in \Z \ \hbox{such that} \ 3+ja=0$$
hence if and only if $a \in \{1,3\}$. This proves (i).

If $(S,\O_S(1)) \cong (\P^1\times \P^1,\O_{\P^1\times\P^1}(a,b))$, suppose, without loss of generality, that $1 \le a \le b$. We have 
\begin{equation}
\label{tang}
T_S \cong (\O_{\P^1}(2) \boxtimes \O_{\P^1}) \oplus (\O_{\P^1} \boxtimes \O_{\P^1}(2)).
\end{equation}
Assume that $T_S$ is aCM. If $a \le 2$, the K\"unneth formula gives
$$0=h^1(T_S(-1))=h^0(\O_{\P^1}(2-a))h^0 (\O_{\P^1}(b-2))+h^0(\O_{\P^1}(a-2))h^0(\O_{\P^1}(2-b))$$
hence $a=b=1$ and then again, the K\"unneth formula gives
$$0=h^1(T_S(-2))=h^0(\O_{\P^1})h^1(\O_{\P^1}(-2))+h^1(\O_{\P^1}(-2))h^0(\O_{\P^1})=2$$
a contradiction. Therefore, $a \ge 3$. 

Vice versa if $a \ge 3$, the K\"unneth formula gives, for every $j \in \Z$,
$$\begin{aligned}
h^1(T_S(j)) = & h^0(\O_{\P^1}(ja+2))h^1(\O_{\P^1}(jb))+h^1(\O_{\P^1}(ja+2))h^0(\O_{\P^1}(jb)) \\
& + h^0(\O_{\P^1}(ja))h^1(\O_{\P^1}(jb+2))+h^1(\O_{\P^1}(ja))h^0(\O_{\P^1}(jb+2)).
\end{aligned}$$
In the above formula, either $j \le -1$ and then $h^0(\O_{\P^1}(ja+2))=h^0(\O_{\P^1}(jb+2))=0$, or $j \ge 0$ and then $h^1(\O_{\P^1}(jb))=h^1(\O_{\P^1}(ja))=0$. Therefore, $T_S$ is aCM.

Finally, it is well known that $H^0(T_S) \ne 0$ when $S$ is $\P^2$ or $\P^1\times \P^1$. Also, in case (i), $H^0(T_S(-1))=H^0(T_{\P^2}(-a))=0$, from the Euler sequence and $a \ge 2$. In case (ii), $H^0(T_S(-1))=H^0(T_{\P^1 \times \P^2}(-a,-b))=0$ by \eqref{tang}. Hence $T_S$ is initialized.
\end{proof}
Note that, with the exception of the case $(\P^2,\O_{\P^2}(2))$, in case $S$ is $\P^2$ or $\P^1 \times \P^1$, $T_S$ is aCM if and only if $K_S+H$ is ample. As we will see, this is confirmed also when $S$ is a Del Pezzo surface of degree at least $5$.

We first need the following

\begin{lemma} 
\label{dpk}
Let $S=S_r, 1 \le r \le 4$ and let $H$ be as in \eqref{h}, \eqref{bi} and \eqref{ma}. Then $T_S$ is not aCM with respect to $H$ if one of the following holds:
\begin{itemize}
\item[(i)] $b_r=1$, or
\item[(ii)] $r=1$ and $b_1+ 1 \le a \le b_1+ 2$, or  
\item[(iii)] $2 \le r \le 4$ and $a=b_1+b_2+1$. 
\end{itemize}
\end{lemma}
\begin{proof}
Given a smooth rational curve $F \subset S$, we will use the exact sequences
\begin{equation}
\label{norm}
0 \to \O_{F}(-F) \to {\Omega_S}_{|F} \to \Omega_F \to 0
\end{equation}
and
\begin{equation}
\label{res}
0 \to \Omega_S(K_S+H-F) \to \Omega_S(K_S+H) \to {\Omega_S}_{|F}(K_S+H) \to 0.
\end{equation}
In several cases we will prove that
\begin{equation}
\label{una}
H^2(\Omega_S(K_S+H-F))=0, (K_S+H)F=0 \ \hbox{and} \ h^1({\Omega_S}_{|F})=1.
\end{equation}
Then \eqref{res} and \eqref{una} will imply that
$$h^1(T_S(-1))=h^1(\Omega_S(K_S+H)) \ge h^1({\Omega_S}_{|F}(K_S+H))=h^1({\Omega_S}_{|F})=1$$
thus $T_S$ is not aCM with respect to $H$. 

We first deal with cases (i) and (iii). Set $F=E_r$ in case (i) and $F \sim \widetilde L-E_1-E_2$ be the strict transform of the line joining $P_1$ and $P_2$ in case (iii). In both cases we have $F^2=K_SF=-1$ and $HF=1$. Now \eqref{norm} is
$$0 \to \O_{\P^1}(1) \to {\Omega_S}_{|F} \to \O_{\P^1}(-2) \to 0$$
from which we get $h^1({\Omega_S}_{|F})=1$. Note that $K_S+H-F$ is either $(a-3)\widetilde L-\sum_{i=1}^{r-1} (b_i-1) E_i-E_r$ in case (i), or $(a-4)\widetilde L-(b_1-2)E_1-(b_2-2)E_2-\sum_{i=3}^r (b_i-1) E_i$ in case (iii). Hence
$$H^2(\Omega_S(K_S+H-F))=0$$ 
by Lemma \ref{dp1} and \eqref{ma}. Thus, \eqref{una} applies and this proves that $T_S$ is not aCM under hypotheses (i) and (iii).  As for (ii), we can assume, by (i) that $b_1 \ge 2$. Let $F \sim \widetilde L-E_1$ be the strict transform of the line through $P_1$, so that $F^2=0$, $K_SF=-2$. If $a=b_1+ 2$ we have $HF=2$ and \eqref{norm} is
$$0 \to \O_{\P^1} \to {\Omega_S}_{|F} \to \O_{\P^1}(-2) \to 0$$
from which we get $h^1({\Omega_S}_{|F})=1$. Now $K_S+H-F=(b_1-2)\widetilde L-(b_1-2) E_1$, hence
$$H^2(\Omega_S(K_S+H-F))=0$$ 
by Lemma \ref{dp1}. Thus, \eqref{una} applies and this proves that $T_S$ is not aCM under hypothesis (ii) when $a=b_1+ 2$. 

Finally assume that $a=b_1+ 1$ and set $F'= \widetilde L$. The exact sequence \eqref{norm} twisted by $(b_1-1)F$ is
$$0 \to \O_{\P^1} \to {\Omega_S}_{|F}((b_1-1)F) \to \O_{\P^1}(-2) \to 0$$
showing that $h^1({\Omega_S}_{|F}((b_1-1)F))=1$. Next, observe that, in the exact sequence
$$0 \to \Omega_S((b_1-2)F) \to \Omega_S((b_1-1)F) \to {\Omega_S}_{|F}((b_1-1)F) \to 0$$
we have $H^2(\Omega_S((b_1-2)F))=0$ by Lemma \ref{dp1}, hence $h^1(\Omega_S((b_1-1)F) \ge h^1({\Omega_S}_{|F}((b_1-1)F))=1$. 

Now, the exact sequence \eqref{norm} for $F'$ and twisted by $(b_1-1)F$ is
$$0 \to \O_{\P^1}(b_1-2) \to {\Omega_S}_{|F'}((b_1-1)F) \to \O_{\P^1}(b_1-3) \to 0$$
hence $H^1({\Omega_S}_{|F'}((b_1-1)F))=0$. Then the exact sequence
$$0 \to \Omega_S((b_1-1)F-F') \to \Omega_S((b_1-1)F) \to {\Omega_S}_{|F'}((b_1-1)F) \to 0$$
implies that
$$h^1(T_S(-1))=h^1(\Omega_S(K_S+H))=h^1(\Omega_S((b_1-1)F-F') \ge h^1(\Omega_S((b_1-1)F) \ge 1.$$ 
Hence $T_S$ is not aCM with respect to $H$ also in case (ii). 
\end{proof}

\renewcommand{\proofname}{Proof of Theorem \ref{acm}}
\begin{proof}
First, assume that $(S,\O_S(1))$ is as in (i)-(iii) of the theorem.

In cases (i)-(ii), we are done by Proposition \ref{min2}. In case (iii), we have that $K_S+H$ is ample by \cite[Cor.~4.6]{dr}. Note that $S$ satisfies Bott vanishing by Remark \ref{bott}. Therefore, if $j \ge 0$ we have, since $T_S \cong \Omega_S(-K_S)$,
$$H^1(T_S(j))=H^1(\Omega_S(-K_S+jH))=0$$
because $-K_S+jH$ is ample. On the other hand, if $j \le -1$ we have, by Serre's duality
$$H^1(T_S(j))=H^1(\Omega_S(K_S-jH))^*=0$$
because $K_S-jH=K_S+H-(j+1)H$ is ample. Thus, $T_S$ is aCM with respect to $H$. As is well known, $H^0(T_S) \ne 0$ with the exception of case (iii) for $r=4$, and $H^0(T_S(-1))=0$, so that $T_S$ is initialized.

Vice versa, assume that $T_S$ is initialized and aCM. 

Then Proposition \ref{aU2} implies that either $S \in \{\P^2, \P^1 \times \P^1\}$ and we are done by Proposition \ref{min2} or
\begin{equation}
\label{dp}
(S,H)=(S_r,a\widetilde L-\sum\limits_{i=1}^r b_iE_i), 1 \le r \le 3
\end{equation}
because for $r=4$ we have that $H^0(T_S)=0$.

In the cases \eqref{dp}, we have the conditions \eqref{ma} and we can assume, without loss of generality, that $b_1 \ge \ldots \ge b_r$. Then Lemma \ref{dpk} implies that $b_r \ge 2$, $a \ge b_1+3$ when $r=1$ and $a \ge b_1+b_2+2$ when $2 \le r \le 3$. It follows that $H$ is as in case (iii) (by adding permutations of the $b_i$'s). 
\end{proof}
\renewcommand{\proofname}{Proof}

\begin{remark} 
\label{tante}
Let $X \subset \P^N$ be a smooth irrreducible $n$-dimensional variety with $H^1(T_X)=0$. Let $A$ be an ample line bundle and let $N$ be a nef line bundle on $X$. Then $T_X$ is aCM with respect to $H=mA+N, m \gg 0$.
\end{remark} 
Indeed, by Fujita's vanishing theorem \cite[Thm.~1.4.35]{laz} there $m_0 \in \Z$ such that $H=mA+N$ is very ample for $m \ge m_0$ and $H^p(T_X(lA+D))=0, H^p(\Omega_X(K_X+lA+D)=0$ for $p > 0, l \ge m_0$ and any nef divisor $D$. Let $m \ge m_0$ and let $i \in \{1, \ldots, n-1\}$. Now, if $j \ge 1$, then $H^i(T_X(j))=H^i(T_X(jmA+jN))=0$ since $jm \ge m \ge m_0$. If $j \le -1$, then $H^i(T_X(j))=H^{n-i}(\Omega_X(K_X-jmA-jN)=0$ since $-jm \ge m \ge m_0$. Thus, $T_X$ is aCM.

\vskip .3cm 
Finally, we have

\renewcommand{\proofname}{Proof of Corollary \ref{au}}
\begin{proof}
It follows immediately from Theorems \ref{wu}, \ref{acm} and Lemma \ref{dpk}.
\end{proof}
\renewcommand{\proofname}{Proof}

\section{Acknowledgement}

The first named author is currently supported by the Post-doctoral Fellowship provided by the National Board for Higher Mathematics (NBHM), Department of Atomic Energy (DAE), Government of India. He gratefully acknowledges the same and also thanks INdAM for the support to a visit to Italy in February 2025. All the authors thank Prof. Francesco Malaspina and Vincenzo Antonelli for hosting them during February 2025 and Politecnico di Torino for their warm hospitality. The initial discussion took place during the conference titled ``Gianfranco Casnati's Legacy: A meeting in Algebraic Geometry".

\appendix
\section{A numerical lemma}

\begin{lemma}
\label{conto}
Let $(a;b_1,b_2,b_3,b_4) \in \Z^5$ be such that $b_1 \ge b_2 \ge b_3 \ge b_4 \ge 1, a \ge b_1+b_2+1$ and
\begin{equation}
\label{1}
a^2-b_1^2-b_2^2-b_3^2-b_4^2-6a+2b_1+2b_2+2b_3+2b_4 \le 0.
\end{equation}
Then $(a;b_1,b_2,b_3,b_4)$ is one of the following:
\begin{enumerate}
\item $(5;1,1,1,1)$. 
\item $(6;2,2,2,2)$.  
\item $(b_1+3;b_1,1,1,1), 1 \le b_1 \le 3$. 
\item $(b_1+3;b_1,2,1,1), 2 \le b_1 \le 3$. 
\item $(8;4,3,3,3)$. 
\item $(b_1+3;b_1,2,2,2), 3 \le b_1 \le 4$.  
\item $(b_1+3;b_1,2,2,1), 3 \le b_1 \le 4$. 
\item $(b_1+2;b_1,1,1,1), b_1 \ge 2$. 
\item $(9;4,4,4,3)$. 
\item $(7;3,3,3,2)$.  
\item $(2b_1+1;b_1,b_1,b_1,1), 2 \le b_1 \le 3$. 
\item $(2b_1+1;b_1,b_1,b_1,b_1), b_1 \ge 1$. 
\end{enumerate}
\end{lemma}
\begin{proof}
\noindent\underline{Case 1: $a \ge b_1+b_2+3$.} 

Then \eqref{1} gives
$$0 \ge 2b_1b_2+2(b_1+b_2+b_3+b_4)-b_3^2-b_4^2-9 \ge 2(b_1+b_2+b_3+b_4)-9$$
that implies $b_1=1$, hence $b_2=b_3=b_4=1$ and $a \ge 5$. But \eqref{1} is $a^2-6a+4 \le 0$, therefore $a=5$ and we are in case (1).

\noindent\underline{Case 2: $a=b_1+b_2+2$.} 

Then \eqref{1} gives
\begin{equation}
\label{2}
0 \ge 2b_1b_2+2b_3+2b_4-b_3^2-b_4^2-8 \ge 2(b_3+b_4)-8
\end{equation}
hence $b_3+b_4 \le 4$ and therefore $b_4 \le 2$. If $b_4=2$ it must be also $b_3=2$ and \eqref{2} becomes $b_1b_2-4 \le 0$, hence also $b_1=b_2=2$ and we are in case (2). Next, assume that $b_4=1$, so that $b_3 \le 3$ and \eqref{2} becomes 
$$0 \ge 2b_1b_2-b_3^2+2b_3-7 \ge b_1b_2+2b_3-7.$$
The above is not possible for $b_2 \ge 2$, hence $b_2=1$, so that $b_3=b_4=1$. Now \eqref{2} becomes $b_1 \le 3$ and we are in case (3).

Therefore, it remains to study the following case.

\noindent\underline{Case 3: $a=b_1+b_2+1$.} 

Then \eqref{1} gives
$$0 \ge 2b_1b_2-2b_1-2b_2-b_3^2-b_4^2+2b_3+2b_4-5$$
that can be rewritten as
\begin{equation}
\label{3}
0 \ge 2(b_1-1)(b_2-1)-(b_3-1)^2-(b_4-1)^2-5.
\end{equation}
\noindent\underline{Subcase 3.1) $b_2>b_3$.} 

Then $b_1-1 \ge b_3, b_2-1 \ge b_4$ and \eqref{3} becomes $0 \ge 2b_3+2b_4-7$, so that $b_4=1, b_3 \le 2$ and \eqref{3} becomes 
\begin{equation}
\label{4}
0 \ge 2(b_1-1)(b_2-1)-(b_3-1)^2-5 \ge (b_1-1)(b_2-1)+2(b_3-1)-4.
\end{equation}
The above is not possible for $b_3=2$, hence $b_3=1$, so that $b_4=1$. Now $b_2 \ge 2$ and, if $b_2 \ge 3$, \eqref{4} is contradicted. Therefore, $b_2=2$, \eqref{4} gives $b_1 \le 3$ and we are in case (4).

Thus, from now on, we will be in

\noindent\underline{Subcase 3.2) $b_2=b_3$.} 

Then \eqref{3} is
\begin{equation}
\label{5}
0 \ge 2(b_1-1)(b_3-1)-(b_3-1)^2-(b_4-1)^2-5.
\end{equation}

\noindent\underline{Subcase 3.2.1) $b_1>b_2$.} 

Then $b_1-1 \ge b_3 \ge b_4$ and replacing in \eqref{5} we get 
$$b_3+b_4 \le 7 \ \hbox{and} \ b_4 \le 3.$$
If $b_4=3$ then $b_3 \le 4$ and equality contradicts \eqref{5}. Therefore, $b_3=3$ and \eqref{5} gives $b_1 \le 4$. Since $b_1>b_2=b_3=3$, we deduce that $b_1=4$ and we are in case (5). If $b_4=2$ we have $2 \le b_3 \le 5$. If $b_3 \ge 3$ we contradict \eqref{5}, hence $b_3=2$ and \eqref{5} gives $b_1 \le 4$. Since $b_1>b_2=b_3=2$, we deduce that $3 \le b_1 \le 4$ and we are in case (6). If $b_4=1$ we have $1 \le b_3 \le 6$. If $b_3 \ge 3$ we contradict \eqref{5}, hence there are two possibilities. If $b_3=2$ the \eqref{5} gives $b_1 \le 4$. Since $b_1>b_2=b_3=2$, we deduce that $3 \le b_1 \le 4$ and we are in case (7). If $b_3=1$ then $b_2=b_3=b_4=1, b_1 \ge 2$ and we are in case (8). Now Subcase 3.2.1 is concluded and we are in

\noindent\underline{Subcase 3.2.2) $b_1=b_2$.} 

Therefore, $b_1=b_2=b_3$ and \eqref{5} becomes
\begin{equation}
\label{6}
0 \ge (b_3-1)^2-(b_4-1)^2-5.
\end{equation}
If $b_3 > b_4$, replacing in \eqref{6} we get $b_4 \le 3$. When $b_4=3$ we have $b_3 \ge 4$ and  \eqref{6} gives $b_3=4$. This is case (9). When $b_4=2$ we have $b_3 \le 3$ by \eqref{6}, hence $b_3=3$ and we are in case (10). When $b_4=1$ we have $b_3 \le 3$ by \eqref{6} we are in case (11). 
Finally, if $b_3=b_4$, we get case (12). 
\end{proof}

\end{document}